\pdfoutput=1
\RequirePackage{ifpdf}
\ifpdf % We~are running pdfTeX in pdf mode
\documentclass[pdftex]{sigma}
\else
\documentclass{sigma}
\fi

\newcommand{\RS}[1]{\textcolor{red}{#1}}

\newtheorem{thm}{Theorem}[section]

\newtheorem{Lemma}[thm]{Lemma}

\newcommand{\R}{\mathbb{R}}

\newcommand{\C}{\mathbb{C}}
\newcommand{\Z}{\mathbb{Z}}
\newcommand{\N}{\mathbb{N}}
\newcommand{\E}{\mathcal{E}}
\newcommand{\Sp}{\mathcal{S}}
\newcommand{\D}{\Delta_\E}
\DeclareMathOperator{\sn}{sn}
\DeclareMathOperator{\cn}{cn}
\DeclareMathOperator{\dn}{dn}
\DeclareMathOperator{\am}{am}

\numberwithin{equation}{section}

\begin{document}

\allowdisplaybreaks

\newcommand{\arXivNumber}{2312.01620}

\renewcommand{\PaperNumber}{067}

\FirstPageHeading

\ShortArticleName{The Laplace--Beltrami Operator on the Surface of the Ellipsoid}

\ArticleName{The Laplace--Beltrami Operator\\ on the Surface of the Ellipsoid}

\Author{Hans VOLKMER}

\AuthorNameForHeading{H.~Volkmer}

\Address{Department of Mathematical Sciences, University of Wisconsin--Milwaukee, USA}
\Email{\href{mailto:volkmer@uwm.edu}{volkmer@uwm.edu}}

\ArticleDates{Received December 04, 2023, in final form July 10, 2024; Published online July 25, 2024}

\Abstract{The Laplace--Beltrami operator on (the surface of) a triaxial ellipsoid admits a~sequence of real eigenvalues diverging to plus infinity. By introducing ellipsoidal coordinates, this eigenvalue problem for a partial differential operator is reduced to a two-parameter regular Sturm--Liouville problem involving ordinary differential operators. This two-parameter eigenvalue problem has two families of eigencurves whose intersection points determine the eigenvalues of the Laplace--Beltrami operator. Eigenvalues are approximated numerically through eigenvalues of generalized matrix eigenvalue problems. Ellipsoids close to spheres are studied employing Lam\'e polynomials.}


\Classification{34B30; 34L15}

\section{Introduction}

In 1839, Lam\'e \cite{L} showed that the Laplace equation
\[ \frac{\partial^2u}{\partial x^2}+\frac{\partial^2 u}{\partial y^2}+\frac{\partial^2 u}{\partial z^2}=0\]
after transformation to ellipsoidal coordinates $\alpha$, $\beta$, $\gamma$ \cite[Section~1.6]{A} can be solved by the method of separation of variables.
The orthogonal coordinate surfaces of ellipsoidal coordinates are ellipsoids, hyperboloids of one sheet and
hyperboloids of two sheets; all confocal.
Lam\'e obtained solutions of product form
\begin{equation*}%\label{product1}
 u(x,y,z)=w_1(\alpha)w_2(\beta)w_3(\gamma),
\end{equation*}
where $w_1$, $w_2$, $w_3$ all satisfy the Lam\'e differential equation
\begin{equation}\label{lame1}
\frac{{\rm d}^2w}{{\rm d}t^2}+\bigl(h-n(n+1)k^2 \sn^2(t,k)\bigr) w=0.
\end{equation}
This equation contains the Jacobian elliptic function $\sn(t,k)$ with modulus $k\in(0,1)$
and two separation constants $h$ and $n(n+1)$.
If $n=0,1,2,\dots$, then equation \eqref{lame1} admits Lam\'e polynomial solutions \cite[Chapter~IX]{A}.
Lam\'e polynomials are nontrivial solutions of \eqref{lame1} that are polynomials in $\sn(t,k), \cn(t,k)$ and $\dn(t,k)$.
For given $n\in\N_0$, there are $2n+1$ special values of~$h$ such that \eqref{lame1} admits a Lam\'e polynomial solution.
If we choose one of these special solutions for $w_1$, $w_2$, $w_3$, then $u$ becomes a harmonic polynomial in the variables $x$, $y$, $z$ also known as an ellipsoidal harmonic.
For the theory of ellipsoidal harmonics, we refer to \cite{D, H, Ho}.

We also mention Mathieu's work \cite{M} on the vibrating elliptic membrane; see also \cite[Section~4.31]{MS} and \cite[p.~4]{Vbook}.
The eigenvalues $\lambda$ are determined from the equation $-\Delta u=\lambda u$ for the two-dimensional Laplace operator $\Delta$ under Dirichlet boundary conditions on an ellipse.
This problem can be solved by the method of separation of variables after transformation to planar elliptic coordinates.
We obtain two Mathieu equations coupled by the eigenvalue parameter~$\lambda$ and an additional separation parameter.

The goal of this paper is to demonstrate that following the method used in these examples we can find all eigenvalues and eigenfunctions of
the Laplace--Beltrami operator $\D$ on the ellipsoid~$\E$ defined by
\begin{equation}\label{ellipsoid}
 \frac{x^2}{a^2}+\frac{y^2}{b^2}+\frac{z^2}{c^2}=1,\qquad a>b>c>0.
 \end{equation}
Here we consider the ellipsoid $\E$ as a two-dimensional analytic Riemannian manifold \cite{J}.
We show that the eigenvalue equation
\begin{equation}\label{problem}
 -\D u=\lambda u
 \end{equation}
after transformation to ellipsoidal coordinates (now denoted by $s$, $t$) can be solved by the method of separation of variables.
We find eigenfunctions of product form
\begin{equation}\label{product2}
 u= v(s) w(t).
\end{equation}
The function $w$ is now a solution of the differential equation
\begin{equation}\label{lame2}
\frac{{\rm d}}{{\rm d}t}\biggl(\frac{1}{q(t)}\frac{{\rm d}w}{{\rm d}t}\biggr)+q(t)\bigl(h-\lambda k^2\sn^2(t,k)\bigr)w=0,
\end{equation}
where \smash{$q(t):=\bigl(a^2\cn^2(t,k)+b^2\sn^2(t,k)\bigr)^{1/2}$},
and $v$ satisfies a similar equation.
We notice that equation \eqref{lame2} reduces to Lam\'e's equation \eqref{lame1}
when $a=b=1$ and $\lambda=n(n+1)$.
We obtain a two-parameter Sturm--Liouville eigenvalue problem
involving two ordinary differential equations
that is fully coupled. This means that the eigenvalue parameter $\lambda$ and the separation parameter $h$
appear in both differential equations. To such a system we
apply the general theory of two-parameter eigenvalue problems as can be found in the books
\cite{AM} and \cite{Vbook}. Since we have only two parameters, the eigenvalues $\lambda$ of $-\D$ are then determined by intersection points of corresponding eigencurves.

The eigenfunctions of the form \eqref{product2} when properly normalized will give us an orthonormal basis
for the Hilbert space $L^2(\E)$.
To the best knowledge of the author, this treatment of equation~\eqref{problem}
is new. Also the treatment of equation \eqref{lame2} appears to be new.

One should point out an important difference between the Laplace--Beltrami operator~$\Delta_\Sp$ on the unit sphere $\Sp$ in $\R^3$ and
the Laplace--Beltrami operator $\Delta_\E$ on the ellipsoid $\E$.
If~$r$,~$\theta$,~$\phi$ denote spherical coordinates in $\R^3$ and $w(\theta,\phi)$ is a spherical harmonic of degree $n\in\N_0$ (that is, $-\Delta_\Sp w=n(n+1)w$), then $r^nw(\theta,\phi)$ is a solution of $\Delta u=0$.
This shows a close connection between the Laplacian in $\R^3$ and $\Delta_\Sp$.
There is no similar relationship between the Laplacian in $\R^3$ and $\Delta_\E$ when we use ellipsoidal coordinates. In particular,
there is no direct relationship between ellipsoidal harmonics and eigenfunctions of~$\Delta_\E$.

In a previous paper \cite{V}, the author studied the eigenvalues of the Laplace--Beltrami operator on a spheroid, i.e., an ellipsoid of revolution.
The eigenvalues of the Laplace--Beltrami operator on a surface of revolution in $\R^3$ can always be found by the method of separation of variables.
We obtain a system of two ordinary differential equation coupled by two parameters, however, in the first equation only
one parameter appears so that this (actually trivial) eigenvalue problem can be solved first.
We are then left with a classical Sturm--Liouville problem for the second equation.
Such a ``weakly'' coupled system is much easier to treat than the fully coupled one.
We point out that it is not possible to simply set $c=b$ in this paper and obtain results from \cite{V}.
In order to obtain results from \cite{V}, we have to consider the limiting case $c\to b$ with $a$, $b$ being fixed.
The precise justification of the limiting process is a delicate one and is not considered in this paper.

We refer to paper \cite{EK} that studies the eigenvalue problem \eqref{problem} when the ellipsoid $\E$ is close to a sphere.
We also mention the related papers \cite{BF} and \cite{MT}.
For some modern applications of eigenfunction expansions for Laplace--Beltrami operators, we refer to the introduction of \cite{EK}.
We add possible applications of fractional powers of Laplace--Beltrami operators to contemporary studies involving ``non-local''
partial differential equations
\cite{GGL}, and also the use of Laplace--Beltrami eigenfunctions expansions on a boundary $\partial\Omega$ in the theory of boundary triples
for linear partial differential operators inside a domain $\Omega$ \cite[Theorem 3.2]{KZ}.
Also note the explicit use of Laplace--Beltrami eigenfunctions expansions in \cite[Section~1 and Theorems 2.2--2.4]{EKa}.

This paper is organized as follows. In Section \ref{coords}, we consider the algebraic and transcendental form of ellipsoidal coordinates for the ellipsoid~$\E$.
The transcendental form has the advantage that it leads us to regular Sturm--Liouville problems
whereas the algebraic form gives us singular Sturm--Liouville problems.
In Section \ref{LB}, we transform equation \eqref{problem}
to ellipsoidal coordinates
and show that it ``separates'', that is, it allows for solutions of the form \eqref{product2}.
In Section \ref{twopara}, we obtain and investigate the two-parameter Sturm--Liouville problem that
completely determines the eigenvalues and eigenfunctions of $-\D$ of even parity.
We show in Theorem \ref{t2} that the system of these eigenfunctions is complete in the Hilbert space of square-integrable functions on~$\E$ with even parity.
In Section \ref{other}, we consider the corresponding two-parameter problems for the seven other parities.
The system of all eigenfunctions is then complete in the Hilbert space of all square-integrable functions on $\E$.

In Section \ref{sphere}, we collect some known results for the Laplace--Beltrami operator on the sphere
in terms of sphero-conal coordinates that will be useful in subsequent sections.
In Section \ref{prufer}, we mention that the two-parameter eigenvalue problem
can be simplified by the Pr\"ufer transformation.
In Section \ref{matrix}, we present an efficient method to compute the eigenvalue of~$-\D$ numerically.
The eigencurves are approximated by the eigenvalues of finite matrix problems.
Finally, in Section \ref{sphere2}, we consider ellipsoids which are close to the unit sphere.
This allows us to make a~connection to the results in~\cite{EK}.

\section{Ellipsoidal coordinates}\label{coords}

We consider the ellipsoid $\E$ given by \eqref{ellipsoid}.
On the ellipsoid $\E$ we introduce (algebraic) ellipsoidal coordinates $\mu$, $\nu$ \cite[Section~15.1.1]{EMO3} by
\begin{gather}
x=a k\mu^{1/2}\nu^{1/2},\label{ell1}\\
y= b k (k')^{-1}(\mu-1)^{1/2}(1-\nu)^{1/2},\label{ell2}\\
z=c (k')^{-1}\bigl(1-k^2\mu\bigr)^{1/2}\bigl(1-k^2\nu\bigr)^{1/2},\label{ell3}
\end{gather}
where $k,k'\in(0,1)$ are determined by
\begin{equation}\label{k}
 k^2=\frac{a^2-b^2}{a^2-c^2},\qquad k'^2=1-k^2= \frac{b^2-c^2}{a^2-c^2}.
 \end{equation}
There is a one-to-one correspondence between the rectangle $0<\nu<1<\mu<k^{-2}$ and
\[ \E_+:=\{(x,y,z)\in\E\colon x,y,z>0\} .\]
There is also a one-to-one correspondence between the closed rectangle ${0\le\nu\le 1\le \mu\le k^{-2}}$
and the closure $\bar{\E}_+$ of $\E_+$.
Indeed, if we start at $(\mu,\nu)=(1,0)$ and move along the boundary of the rectangle, then
$(x,y,z)$ moves along the boundary of $\E_+$ starting at $(0,0,c)$. The points $(\mu,\nu)=(1,1)$, $(\mu,\nu)=\bigl(k^{-2},1\bigr)$, $(\mu,\nu)=\bigl(k^{-2},0\bigr)$ correspond to
$(x,y,z)=(ak,0,ck')$, $(x,y,z)=(a,0,0)$, $(x,y,z)=(0,b,0)$, respectively.
The functions $\mu$ and $\nu$ can now be extended to continuous functions on $\E$ with parity $(0,0,0)$,
that is, functions that satisfy
\[ f(x,y,z)=f(-x,y,z)=f(x,-y,z)=f(x,y,-z)\qquad \text{for all} \quad (x,y,z)\in\E.\]
We note that these functions $\mu$ and $\nu$ also satisfy
\begin{equation}\label{munu}
\mu\nu=a^{-2}k^{-2} x^2,\qquad \mu+\nu=1+a^{-2}k^{-2} x^2+b^{-2}k^{-2}(k')^2y^2 .
\end{equation}
It follows that $\mu\nu$ and $\mu+\nu$ are analytic functions on $\E$. Therefore, every symmetric polynomial
in~$\mu$,~$\nu$ is also an analytic function on $\E$. However, $\mu$, $\nu$ are not analytic on $\E$.
In fact, the functions~$\mu$,~$\nu$ are analytic at every point of $\E$ except the four points
with $x=\pm ak$, $y=0$. This follows from the observation that we can solve \eqref{munu} for $\mu$, $\nu$
locally at $(x_0,y_0)$ by analytic functions of $x$, $y$ except when $\mu(x_0,y_0,z_0)=\nu(x_0,y_0,z_0)$.
But $\mu=\nu$ only happens when $\nu=\mu=1$ and this happens only at the four points $(x_0,y_0,z_0)=(\pm ak,0,\pm ck')$.

\begin{Lemma}\label{l1}
Let $f$ be an analytic function defined on an open region $D$ in the complex plane containing the interval $\bigl[0,k^{-2}\bigr]$.
Then $f(\mu)f(\nu)$ is an analytic function on $\E$.
\end{Lemma}
\begin{proof}
We already saw that $f(\mu)f(\nu)$ is analytic on $\E$ with the possible exception of the four points
$(x_0,y_0,z_0)=(\pm ak,0,\pm ck')$.
In order to show that $f(\mu)f(\nu)$ is also analytic at these four points, it will be
sufficient to show analyticity at $(x_0,y_0,z_0)=(ak,0,ck')$.
The function $f$ admits a power series expansion
\[ f(\xi)=\sum_{n=0}^\infty c_n (\xi-1)^n \]
for $\xi$ close to $1$.
Then
\begin{equation}\label{ff}
 f(\xi)f(\eta)=\sum_{n=0}^\infty f_n(\xi,\eta),\qquad f_n(\xi,\eta):=\sum_{m=0}^n c_mc_{n-m} (\xi-1)^m(\eta-1)^{n-m}
 \end{equation}
for $\xi$ and $\eta$ close to $1$.
The functions $f_n(\xi,\eta)$ are symmetric polynomials in $\xi,\eta$. Therefore, the functions
$f_n(\mu,\nu)$ are analytic in $x$, $y$
for $(x,y)$ close to $(x_0,y_0)$, where we allow complex values of $x$, $y$.
Since the expansion in \eqref{ff} converges locally uniformly in a neighborhood of $(x_0,y_0)$, we obtain
that $f(\mu)f(\nu)$ is analytic at $(x_0,y_0,z_0)$.
\end{proof}

The metric tensor of ellipsoidal coordinates $\mu\in\bigl(1,k^{-2}\bigr)$, $\nu\in(0,1)$ on $\E_+$ is given by
\begin{gather}
 g_1 := g_{11}=(\mu-\nu) F(\mu),\label{gg1}\\
 g_2 := g_{22}=-(\mu-\nu) F(\nu),\label{gg2}\\
 g_{12} = 0,\label{gg3}
\end{gather}
where
\[ F(\mu):=\frac{\bigl(a^2-b^2\bigr)\mu-a^2}{4\mu(\mu-1)\bigl(\mu-k^{-2}\bigr)} .\]
In particular, this shows that these coordinates are orthogonal.

Let $\sn(\xi,k)$, $\cn(\xi,k)$ and $\dn(\xi,k)$ denote the Jacobian elliptic functions
corresponding to the modulus $k\in(0,1)$ given by \eqref{k}.
These functions have period $4K$, where $K$ is the complete elliptic integral
\[ K=K(k)=\int_0^{\pi/2} \frac{{\rm d}\phi}{\bigl(1-k^2\sin^2\phi\bigr)^{1/2}}.\]
We will also use Jacobian elliptic functions corresponding to the modulus $k'=\sqrt{1-k^2}$ and the
complete elliptic integral $K'=K(k')$.
By setting
\[ \mu=k^{-2}\dn^2(s,k'),\qquad \nu=\sn^2(t,k),\qquad s\in(0,K'),\quad t\in(0,K),\]
we introduce transcendental ellipsoidal coordinates $s$, $t$ in $\E_+$, so
\begin{gather*}
x = a\dn(s,k')\sn(t,k),\qquad
y = b\cn(s,k')\cn(t,k),\qquad
z = c\sn(s,k')\dn(t,k).
\end{gather*}
When we allow $s\in(-K',K')$ and $t\in(-2K,2K]$, then the ellipsoidal coordinates $(s,t)$ are in one-to-one correspondence
with the points of the ellipsoid $\E$ with the exception of two cuts passing through the north pole $(0,0,c)$ and south pole
$(0,0,-c)$ given by $s=K'$ and $s=-K'$, respectively, that is,
\[ x=ak\sn(t,k),\qquad y=0,\qquad z =\pm c \dn(t,k),\qquad t\in[-K,K] .\]
Figure \ref{fig1} depicts the ellipsoid with semi-axes $a=3$, $b=2$, $c=1$ in bird's eye view.
Coordinate lines $s=\frac13mK'$, $m=0,1,2$ and $t=\frac13 nK$, $n=-5,-4,\dots,5,6$
are shown in red and green, respectively. The branch cut is shown in blue.

\begin{figure}[t]
\centering
\includegraphics[width=10cm]{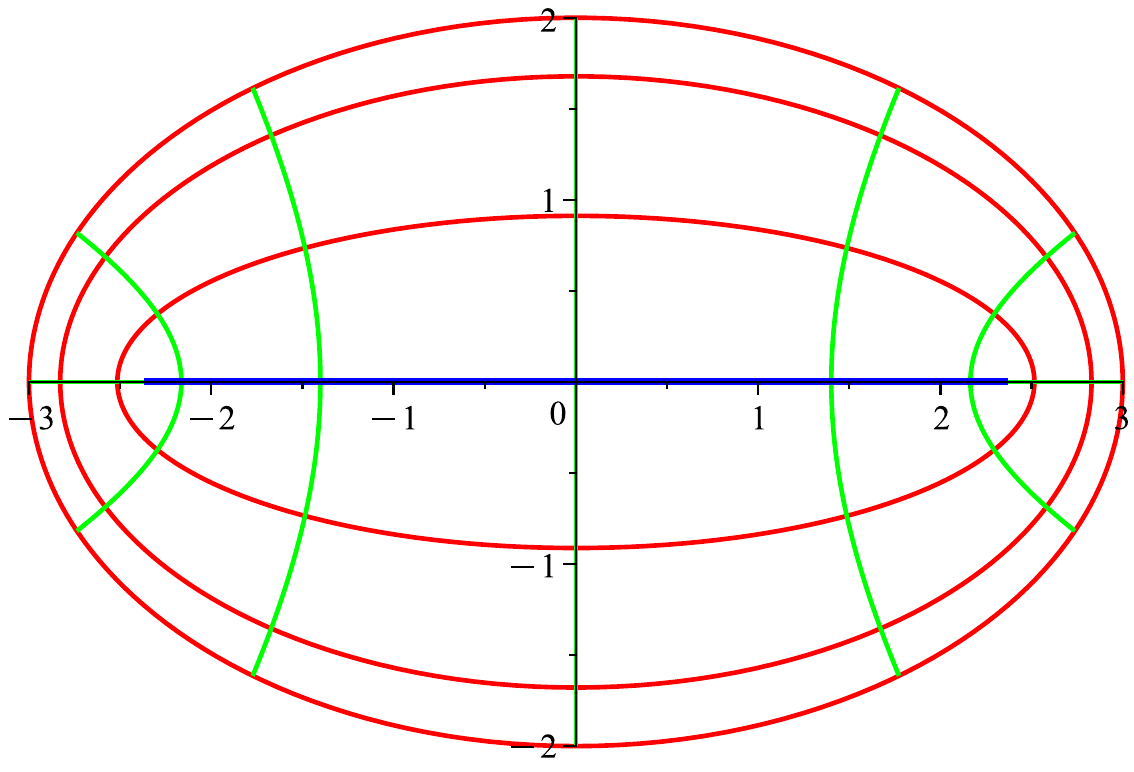}
\caption{Coordinate lines of ellipsoidal coordinates\label{fig1}.}
\end{figure}

From \eqref{gg1}--\eqref{gg3}, we find the metric tensor in terms of the coordinates $s$, $t$
\begin{gather}
 g_1 = \bigl(c^2\cn^2(s,k')+b^2\sn^2(s,k')\bigr)\bigl(\dn^2(s,k')-k^2\sn^2(t,k)\bigr),\label{g1}\\
 g_2 = \bigl(a^2\cn^2(t,k)+b^2\sn^2(t,k)\bigr)\bigl(\dn^2(s,k')-k^2\sn^2(t,k)\bigr),\label{g2}\\
 g_{12} = 0.
\end{gather}

\section{The Laplace--Beltrami operator}\label{LB}

In terms of the coordinates $\mu$ and $\nu$, the Laplace--Beltrami operator $\D$ on the ellipsoid $\E$ is given~by
\begin{align*}
 \D u&{}=(g_1g_2)^{-1/2}\biggl(\frac{\partial}{\partial \mu}\biggl(g_2^{1/2}g_1^{-1/2}
\frac{\partial u}{\partial \mu}\biggr)+\frac{\partial}{\partial \nu}\biggl(g_1^{1/2}g_2^{-1/2}
\frac{\partial u}{\partial \nu}\biggr) \biggr)\\
&{}=\frac1{\mu-\nu}\biggl(\frac{1}{F(\mu)^{1/2}}\frac{\partial}{\partial \mu}\biggl(\frac{1}{F(\mu)^{1/2}}\frac{\partial u}{\partial \mu}
\biggr)
+\frac{1}{(-F(\nu))^{1/2}}\frac{\partial}{\partial \nu}\biggl(\frac{1}{(-F(\nu))^{1/2}}\frac{\partial u}{\partial \nu}\biggr)\biggr).
\end{align*}
In the eigenvalue equation $-\D u=\lambda u$, we can separate variables as follows.
If $v(\mu)$ is a solution of the differential equation
\begin{equation}\label{ode1}
\frac{{\rm d}}{{\rm d}\mu}\biggl(\frac{1}{F(\mu)^{1/2}}\frac{{\rm d}v}{{\rm d}\mu}\biggr)+F(\mu)^{1/2}\bigl(\lambda\mu-hk^{-2}\bigr)v =0,\qquad 1<\mu<k^{-2},
\end{equation}
and $w(\nu)$ is a solution of the differential equation
\begin{equation}\label{ode2}
\frac{{\rm d}}{{\rm d}\nu}\biggl(\frac{1}{(-F(\nu))^{1/2}}\frac{{\rm d}w}{{\rm d}\nu}\biggr)-(-F(\nu))^{1/2}\bigl(\lambda \nu-hk^{-2}\bigr) w =0,\qquad 0<\nu<1,
\end{equation}
then $u(\mu,\nu)=v(\mu)w(\nu)$ is a solution of $-\D u=\lambda u$ on $\E_+$. In these equations $h\in\R$ denotes the separation parameter.

Equation \eqref{ode1} can be written in the form
\begin{equation}\label{ode3}
\frac{{\rm d}^2v}{{\rm d}\mu^2}+\frac12\biggl(\frac1\mu+\frac1{\mu-1}+\frac1{\mu-k^{-2}}-\frac1{\mu-d}\biggr)
\frac{{\rm d}v}{{\rm d}\mu}+\bigl(\lambda \mu-hk^{-2}\bigr)F(\mu) v=0,
\end{equation}
where
\[ d:=\frac{a^2}{a^2-b^2}>k^{-2}.\]
Equation \eqref{ode2} can be written in exactly the same form with $\nu$ replacing $\mu$ and $w$ replacing~$v$.
Therefore, the differential equations \eqref{ode1} and \eqref{ode2} are actually the same equation but they are considered on different intervals.
It is interesting to note that equation \eqref{ode3} has five singularities at $0$, $1$, $k^{-2}$, $d$ and $\infty$.
The first four are regular singularities \cite[Sections~5 and~4]{O} with indices $\{0,\frac12\}$, $\{0,\frac12\}$, $\{0,\frac12\}$, $\{0,\frac32\}$
respectively but $\infty$ is an irregular singularity unless $\lambda=0$. Equation~\eqref{ode3} can be considered as an extension of
the Lam\'e equation \eqref{lame1} in its algebraic form \cite[Section~29.2.2]{dlmf}. The Lam\'e equation has only four singularities at $0$, $1$, $k^{-2}$, $\infty$ and all of them are regular. The additional singularity $d$ appears in connection with the Laplace--Beltrami operator on the ellipsoid and makes the differential equation more difficult to treat.

Since the singularities $0$, $1$, $k^{-2}$ are regular and each has index $0$, equation \eqref{ode3}
admits nontrivial Fuchs--Frobenius power series solutions in powers of $\mu$, $\mu-1$ and $\mu-k^{-2}$.
Usually, these solutions are not connected by analytic continuation.
We are interested in the exceptional case that all three solutions are connected.

\begin{Lemma}\label{l2}
Let the parameters $\lambda$, $h$ be such that differential equation \eqref{ode3}
admits a nontrivial analytic solution $f(\mu)$ on an open region $D$ containing the interval $\bigl[0,k^{-2}\bigr]$.
Then the function $u$ on $\E$ defined by $f(\mu)f(\nu)$ is an analytic function on $\E$ with parity $(0,0,0)$ and it is an eigenfunction of $-\D$
corresponding to the eigenvalue $\lambda$.
\end{Lemma}
\begin{proof}
It follows from Lemma \ref{l1} that $u$ is analytic on $\E$.
By our derivation, we know that~$u$ solves $-\D u=\lambda u$ on $\E_+$.
Both $u$ and $\D u$ have the same parity $(0,0,0)$ so $-\D u=\lambda u$ holds~on~$\E$.
\end{proof}

Note that in the situation of Lemma \ref{l2} the region $D$ can be taken as the entire complex plane with the branch cut $[d,\infty)$ removed.

\section{A two-parameter Sturm--Liouville problem}\label{twopara}

Our goal is to prove existence of pairs $\lambda$, $h$ as indicated in Lemma \ref{l2}, and to show
that the corresponding $\lambda$'s comprise all eigenvalues of $-\D$ having eigenfunction with parity $(0,0,0)$.
To that purpose, we prefer to use the transcendental ellipsoidal coordinates $s$, $t$ because then we arrive at
a regular two-parameter Sturm--Liouville system.
If we use the coordinates~$\mu$ and~$\nu$, then our Sturm--Liouville problems have singular endpoints which
makes theoretical and numerical work more difficult.

We set $\mu=k^{-2}\dn^2(s,k')$ and $\nu=\sn^2(t,k)$ in \eqref{ode3}, and, after some computations, we obtain
\begin{gather}\label{SL1}
\frac{{\rm d}}{{\rm d}s}\biggl(\frac{1}{p(s)}\frac{{\rm d}v}{{\rm d}s}\biggr)+p(s)\bigl(-h+\lambda \dn^2(s,k')\bigr) v=0,\\
\frac{{\rm d}}{{\rm d}t}\biggl(\frac{1}{q(t)}\frac{{\rm d}w}{{\rm d}t}\biggr)+q(t)\bigl(h-\lambda k^2\sn^2(t,k)\bigr)w=0, \label{SL2}
\end{gather}
where
\begin{gather}
 p(s) := \bigl(c^2\cn^2(s,k')+b^2\sn^2(s,k')\bigr)^{1/2},\label{p}\\
 q(t) := \bigl(a^2\cn^2(t,k)+b^2\sn^2(t,k)\bigr)^{1/2}\label{q}.
\end{gather}
When we substitute $t=K+ {\rm i}K'-{\rm i}s$, $s\in\R$, and use the identities
\[ k\sn(t,k)=\dn(s,k'),\qquad k\cn(t,k)=-{\rm i}k'\cn(s,k'),\qquad \dn(t,k)=k'\sn(s,k'),\]
we find that $p(s)=q(t)$ and equation \eqref{SL1} transforms to \eqref{SL2}.
So we may say that \eqref{SL1} and~\eqref{SL2} are the same differential equations but considered on different intervals
in the complex plane.
Note also that equation \eqref{SL1} can be obtained from \eqref{SL2} by interchanging $a$ with $c$ (which automatically
interchanges $k$ with $k'$) and replacing $h$ by $\lambda-h$.
If we put $q(t)=1$, then \eqref{SL2} is the Lam\'e equation \eqref{lame1} in its standard form, so \eqref{SL2} can be seen as a generalization of the Lam\'e equation.
We add the Neumann boundary conditions
\begin{gather}\label{bc1}
v'(0) = v'(K')=0, \\
w'(0) = w'(K)=0.\label{bc2}
\end{gather}
They
correspond to the conditions that $v$ as a function of $\mu$ is a Fuchs--Frobenius solution of \eqref{ode3} at $\mu=1$ and $\mu=k^{-2}$ belonging to the indices $0$, and $w$ as a function of $\nu$ is a Fuchs--Frobenius solution of \eqref{ode3}
at $\nu=0$ and $\nu=1$ belonging to the indices $0$.

For every fixed real $\lambda$, equation \eqref{SL1} subject to boundary conditions \eqref{bc1} poses
a regular Sturm--Liouville problem \cite[Section~27]{W} with spectral parameter $-h$.
Similarly, for every fixed real $\lambda$, \eqref{SL2} subject to boundary conditions \eqref{bc2} poses
a regular Sturm--Liouville problem with spectral parameter $h$.
If we combine these problems, we obtain a regular two-parameter Sturm--Liouville problem \cite{AM}.
A pair $\lambda$, $h$ is called an
{\it eigenvalue} of this problem if there exist a~nontrivial solution $v(s)$, $0\le s\le K'$, of~\eqref{SL1} and~\eqref{bc1}
and a nontrivial solution $w(t)$, $0\le t\le K$, of~\eqref{SL2} and~\eqref{bc2}.
If $\lambda$, $h$ is an eigenvalue of this two-parameter Sturm--Liouville problem, then Lemma~\ref{l2}
shows that $\lambda$ is an eigenvalue of $-\D$ having an eigenfunction of parity $(0,0,0)$.

If $v(s)$ is a solution of equation \eqref{SL1}, then $v(s+2K')$ and $v(-s)$ are also solutions.
Therefore, a solution $v(s)$ of \eqref{SL1} satisfies the boundary conditions \eqref{bc1} if and only if
it is even and has period $2K'$. A similar remark applies to equation \eqref{SL2}.
Taking into account the substitution $t=K+{\rm i}K'-{\rm i}s$, we see that $\lambda$, $h$ is an eigenvalue pair if and only if equation \eqref{SL2}
admits a~nontrivial analytic solution $w(t)$ defined on the union $S_1\cup S_2$ of two strips
\[ S_1:=\{t\in\C\colon |{\operatorname{Im}  t}|<\delta\},\qquad S_2=\{t\in\C\colon |{\operatorname{Re}  t-K}|<\delta\}\]
for some sufficiently small positive $\delta$ such that $w(t)$ has period $2K$ on $S_1$, period $2{\rm i}K'$ on $S_2$ and, in addition, is even with respect to $K$, that is, $w(t)=w(2K-t)$. So the eigenvalue problem consists in finding
doubly-periodic solutions of \eqref{SL2}.

For every $\lambda\in\R$, the eigenvalues of the regular Sturm--Liouville problem \eqref{SL2} and \eqref{bc2}
form an increasing sequence
\[ h_0(\lambda)<h_1(\lambda)<h_2(\lambda)<\cdots ,\]
which diverges to $\infty$.
The subscript $n$ of $h_n(\lambda)$ denotes the number of zeros of a corresponding eigenfunction $w_n(t)$ in the interval $(0,K)$.
The graph of the eigenvalue functions $h_n(\lambda)$ are called {\it eigencurves} \cite[Section~6]{AM}.
Similarly, equation \eqref{SL1} (with $h$ replaced by $H$) subject to conditions~\eqref{bc1}
generates eigenvalue functions
\[ H_0(\lambda)>H_1(\lambda)>H_2(\lambda)>\cdots,\]
where the subscript $m$ of $H_m(\lambda)$ denotes the number of zeros of an eigenfunction $v_m(s)$ in the interval $(0,K')$.
The eigenvalues $\lambda$, $h$ of the two-parameter eigenvalue problem are exactly the intersection points of
eigencurves $H_m(\lambda)$ and $h_n(\lambda)$, $m,n\in\N_0$.
A few eigencurves are shown in~Figure~\ref{fig2} in~Section~\ref{matrix}.

We mention some properties of eigencurves proved in \cite{BV}.
We cannot directly apply the results from \cite{BV} to \eqref{SL1} and \eqref{SL2}
because in \cite{BV} it is assumed that the factor of $h$ is identically one. However, we can transform
this factor to $1$ by a substitution of the independent variable. For example, in \eqref{SL2} we
introduce a new variable $\tau$ by setting
\[ \frac{{\rm d}\tau}{{\rm d}t}=q(t) .\]
Then \eqref{SL2} becomes
\begin{equation}\label{ode6}
 -\frac{{\rm d}^2w}{{\rm d}\tau^2}=(h+\lambda r(\tau)) w,\qquad r(\tau)=-k^2\sn^2(t,k).
 \end{equation}
Equation \eqref{ode6} has the form treated in \cite{BV}. Actually, equation \eqref{ode6} looks simpler than \eqref{SL2}
but, unfortunately, we do not have an explicit formula for $r(\tau)$ in terms of $\tau$.

By \cite[Theorem 2.1]{BV}, $H_m(\lambda)$ and $h_n(\lambda)$ are analytic functions of $\lambda$. Their first derivatives \cite[formula~(2.5)]{BV} are given by
\begin{gather*}
H_m'(\lambda)=\frac{\int_0^{K'} p(s)\dn^2(s,k') v_m(s)^2\,{\rm d}s}{\int_0^{K'} p(s)v_m(s)^2\,{\rm d}s},\qquad
h_n'(\lambda)=\frac{\int_0^{K} q(t)k^2\sn^2(t,k) w_n(t)^2\,{\rm d}t}{\int_0^{K} q(t)w_n(t)^2\,{\rm d}t}.
\end{gather*}
It follows from $k^2<\dn^2(s,k')<1$ on $(0,K')$ and $0<\sn^2(t,k)<1$ on $(0,K)$ that
\begin{equation}\label{estder}
0<h_n'(\lambda)<k^2<H_m'(\lambda)<1\qquad \text{for all} \quad \lambda\in\R.
\end{equation}
Regarding the asymptotic behavior of the eigencurves as $\lambda\to\infty$, we have from \cite[Theorem 2.2]{BV}%
\begin{equation}\label{asy}
\lim_{\lambda\to\infty} \frac{H_m(\lambda)}{\lambda}=1,\qquad \lim_{\lambda\to\infty} \frac{h_n(\lambda)}{\lambda}=0.
\end{equation}
If $\lambda=0$, then we can solve \eqref{SL1} and \eqref{SL2} explicitly.
The solution of \eqref{SL1} with $\lambda=0$ and initial conditions $v(0)=1$, $v'(0)=0$ is
\[ v(s)=\cos\biggl((-h)^{1/2}\int_0^s p(\sigma)\,{\rm d}\sigma\biggr).\]
Using this formula and a similar one for \eqref{SL2}, we obtain
\begin{gather}\label{hm0}
 H_m(0)=-m^2\pi^2 \biggl(\int_0^{K'} p(s)\,{\rm d}s\biggr)^{-2} \qquad \text{for}\quad m\in\N_0,\\
h_n(0)=n^2\pi^2 \biggl(\int_0^{K} q(t)\,{\rm d}t\biggr)^{-2} \qquad \text{for} \quad n\in\N_0.\label{Hn0}
\end{gather}

\begin{thm}\label{t1}
For every pair $m,n\in\N_0$, there is exactly one intersection point $(\lambda_{m,n},h_{m,n})$ of the
eigencurves $H_m(\lambda)$ and $h_n(\lambda)$. We have $\lambda_{0,0}=h_{0,0}=0$ and $\lambda_{m,n}>0$, $h_{m,n}>0$ for all $(m,n)\ne (0,0)$.
\end{thm}
\begin{proof}
Let $m,n\in\N_0$. It follows from \eqref{estder} that there is at most one $\lambda$ such that ${H_m(\lambda)\!=\!h_n(\lambda)}$.
By \eqref{hm0} and~\eqref{Hn0} we know that $H_m(0)\le 0\le h_n(0)$. Therefore, by \eqref{asy}, there exists $\lambda$ such that $H_m(\lambda)=h_n(\lambda)$. If $m=n=0$, then $\lambda=0$ and $H_0(0)=h_0(0)=0$. Otherwise, $\lambda>0$ and~$h_n(\lambda)>0$.
\end{proof}

Our two-parameter Sturm--Liouville problem is right-definite, that is,
\setlength{\arraycolsep}{1mm}
\[ D(s,t):=\begin{vmatrix} p(s)\dn^2(s,k') & -p(s) \\ -q(t) k^2\sn^2(t,k) & q(t)\end{vmatrix}=p(s)q(t)\bigl(\dn^2(s,k')-k^2\sn^2(t,k)\bigr)>0 \]
for all $s\in[0,K']$ and $t\in[0,K]$ with the exception of $s=K'$, $t=K$, where the determinant vanishes.
It is well known \cite[Theorem 5.5.1]{AM} that right-definiteness implies the existence and uniqueness of the intersection points $(\lambda_{m,n},h_{m,n})$ of eigencurves as stated in Theorem \ref{t1}.

Let $v_{m,n}(s)$ and $w_{m,n}(t)$ denote real-valued eigenfunctions
satisfying \eqref{SL1}, \eqref{SL2}, \eqref{bc1}, \eqref{bc2} for $\lambda=\lambda_{m,n}$ and $h=h_{m,n}$.
It is easy to prove \cite[Section 3.5]{AM} that the system
of products
\[ u_{m,n}(s,t):=v_{m,n}(s)w_{m,n}(t),\qquad m,n\in\N_0,\]
is orthogonal with respect to the
inner product
\[
\langle f,g\rangle=\int_0^K \int_0^{K'} D(s,t) f(s,t)\overline{g(s,t)}\,{\rm d}s{\rm d}t.
\]
We notice that
\[ D(s,t)=g_1^{1/2}g_2^{1/2} \]
if $g_1$ and $g_2$ are expressed in ellipsoidal coordinates $s$, $t$ according to \eqref{g1} and~\eqref{g2}.
This shows that the eigenfunctions $u_{m,n}$ of $-\D$ are orthogonal
in $L^2(\E)$ because the surface measure of $\E$ has the form ${\rm d}S =g_1^{1/2}g_2^{1/2}\,{\rm d}s {\rm d}t$.
We normalize the functions $u_{m,n}$ so that they have norm~$1$ in $L^2(\E)$.

We have the following completeness result.

\begin{thm}\label{t2}\quad
\begin{itemize}\itemsep=0pt
\item[$(i)$] The double sequence $u_{m,n}$, $m,n\in\N_0$,
forms an orthonormal basis for the subspace of~$L^2(\E)$ consisting of functions with parity $(0,0,0)$.

\item[$(ii)$]
Every number $\lambda_{m,n}$, $m,n\in\N_0$, is an eigenvalue of the Laplace--Beltrami operator $-\D$
with corresponding eigenfunction $u_{m,n}$ of parity $(0,0,0)$.
Every eigenvalue of $-\D$ with an eigenfunction of parity $(0,0,0)$ is equal to one of the $\lambda_{m,n}$
\end{itemize}
\end{thm}
\begin{proof}
(i) This follows from \cite[Theorem 6.8.3]{Vbook}.

(ii)
We already mentioned that each $\lambda_{m,n}$ is an eigenvalue of $-\D$ with eigenfunction $u_{m,n}$.
If~$-\D$ had an eigenvalue $\lambda$ with an eigenfunction $u$ of parity $(0,0,0)$ such that
$\lambda$ is different from any $\lambda_{m,n}$, then $u$ would be orthogonal to each $u_{m,n}$, and this is impossible by
Theorem~\ref{t2}\,(i).
\end{proof}

\section{Eigenfunctions of other parities}\label{other}

Let $\kappa_1,\kappa_2,\kappa_3\in\{0,1\}$. We say that a function $f\colon \E\to\C$ has parity $(\kappa_1,\kappa_2,\kappa_3)$
if
\[ f(x,y,z)=(-1)^{\kappa_1}f(-x,y,z)=(-1)^{\kappa_2}f(x,-y,z)=(-1)^{\kappa_3}f(x,y,-z).\]
Since the Laplace--Beltrami operator $\D$ leaves the parity of a function invariant, it is
sufficient to look for eigenvalues of $-\D$ with eigenfunctions of a given parity.
This splits the eigenvalue problem for $-\D$ into eight subproblems.
We already proved some results on the eigenvalues of $-\D$ with eigenfunctions of parity $(0,0,0)$.
We now mention analogous results for the other seven parities.

The differential equations \eqref{ode3}, \eqref{SL1} and \eqref{SL2} stay the same in all eight cases only the boundary conditions \eqref{bc1} and \eqref{bc2} change.
We need a nontrivial solution of \eqref{ode3} of the form
\[ \mu^{\kappa_1/2}(\mu-1)^{\kappa_2/2}\bigl(\mu-k^{-2}\bigr)^{\kappa_3/2} g(\mu),\]
where $g$ is an analytic function on an open region containing the interval $\bigl[0,k^{-2}\bigr]$, that is,
a~nontrivial solution of \eqref{ode3} which is simultaneously a Fuchs--Frobenius solution at the points~$0$,~$1$,~$k^{-2}$
belonging to the indices $\kappa_1/2$, $\kappa_2/2$, $\kappa_3/2$, respectively.
These conditions translate to boundary conditions for \eqref{SL1} and \eqref{SL2}
as follows:
\begin{alignat}{3}
&v'(0)=0 \quad \text{if $\kappa_3=0$}, && v(0)=0 \quad \text{if $\kappa_3=1$}\label{bc3}, &\\
&v'(K')=w'(K)=0 \quad \text{if $\kappa_2=0$},\qquad && v(K')=w(K)=0 \quad \text{if $\kappa_2=1$}\label{bc4}, &\\
&w'(0)=0 \quad \text{if $\kappa_1=0$}, && w(0)=0 \quad \text{if $\kappa_1=1$}.\label{bc5} &
\end{alignat}
For each $\kappa=(\kappa_1,\kappa_2,\kappa_3)\in\{0,1\}^3$, we obtain a regular two-parameter Sturm--Liouville problem for the differential equations \eqref{SL1} and \eqref{SL2} subject to boundary conditions \eqref{bc3}--\eqref{bc5}.
We denote the corresponding eigenvalue functions by $H_{m,\kappa}(\lambda;a,b,c)$ and $h_{n,\kappa}(\lambda;a,b,c)$.
The solution~$\lambda$ of $H_{m,\kappa}(\lambda;a,b,c)=h_{n,\kappa}(\lambda;a,b,c)$
we denote by $\lambda_{m,n,\kappa}(a,b,c)$. This number $\lambda_{m,n,\kappa}(a,b,c)$ is an eigenvalue of the Laplace--Beltrami operator $-\D$, and every eigenvalue of $-\D$ is of this form.
We recall that $a$, $b$, $c$ denote the semi-axes of the ellipsoid $\E$, $\kappa$ is the parity of a corresponding eigenfunction
$u_{m,n,\kappa}=v_{m,\kappa}(s)w_{n,\kappa})$ when considered as a function on $\E$,
$m$ is the number of zeros of $v_{m,\kappa}$ in $(0,K')$, and $n$ is the number of zeros of $w_{n,\kappa}$ in $(0,K)$.

For each parity, we get analogues of Theorems~\ref{t1} and~\ref{t2}.
We state the completeness theorem below. The proof of this theorem is very similar to the proof of Theorem \ref{t2} and is omitted.

\begin{thm}%\label{t3}
Let $\kappa\in \{0,1\}^3$.
\begin{itemize}\itemsep=0pt
\item[$(i)$]
The sequence $u_{m,n,\kappa}$, $m,n\in\N_0$,
forms an orthonormal basis for the subspace of $L^2(\E)$ consisting of functions with parity $\kappa$.
When we combine all eigenfunctions $u_{m,n,\kappa}$, ${m,n\in\N_0}$, $\kappa\in\{0,1\}^3$, then we obtain an orthonormal basis for
$L^2(\E)$.

\item[$(ii)$]
Every number $\lambda_{m,n,\kappa}$, $m,n\in\N_0$, is an eigenvalue of the Laplace--Beltrami operator $-\D$
with corresponding eigenfunction $u_{m,n,\kappa}$ of parity $\kappa$.
Every eigenvalue of $-\D$ with an eigenfunction of parity $\kappa$ is equal to one of the $\lambda_{m,n,\kappa}$
\end{itemize}
\end{thm}

\section{The Laplace--Beltrami operator on the sphere}\label{sphere}

It is of interest to compare
the two-parameter Sturm--Liouville problem \eqref{SL1} and \eqref{SL2}
subject to boundary conditions \eqref{bc3}--\eqref{bc5} with a well-known simpler problem
that we obtain when treating the Laplace--Beltrami operator $\Delta_\Sp$ on the sphere
\[ \Sp=\{(x,y,z)\colon x^2+y^2+z^2=1\}.\]
It is well known that the operator $-\Delta_\Sp$ has eigenvalues $n(n+1)$, $n\in\N_0$. The eigenspace corresponding to
the eigenvalue $n(n+1)$ has dimension $2n+1$. Eigenfunctions belonging to the eigenvalue $n(n+1)$ are called spherical (surface) harmonics.
Usually, we work with a basis of spherical harmonics that are found by the method of separation variables in spherical coordinates.
We obtain spherical harmonics that are products of associated Legendre polynomials and trigonometric functions.
However, in this work, it is more convenient to
use sphero-conal coordinates. Then we obtain spherical harmonics expressed as products of Lam\'e polynomials known as sphero-conal harmonics.
This is explained below in more detail.

We introduce sphero-conal coordinates $\mu$, $\nu$ on $\Sp$ by \eqref{ell1}--\eqref{ell3} with $a=b=c=1$.
It is important to note that in case of the ellipsoid $\E$ the parameter $k$ is defined by \eqref{k}
whereas $k\in(0,1)$ is arbitrary in case of the sphere.
We then proceed as we did for the Laplace--Beltrami operator on the ellipsoid.
The differential equation \eqref{ode3} becomes the Lam\'e equation~\cite[formula~(29.2.2)]{dlmf} in its algebraic form,
and \eqref{SL2} becomes the Lam\'e equation in its Jacobian form.
We obtain
a~two-parameter Sturm--Liouville problem consisting of differential equations \eqref{SL1} and \eqref{SL2}
with $p(s)=q(t)=1$ subject to boundary conditions \eqref{bc3}--\eqref{bc5}.
We will denote the corresponding eigenvalue functions by $H_{m,\kappa}(\lambda,k)$, $h_{n,\kappa}(\lambda,k)$
and the solution $\lambda$ of $H_{m,\kappa}(\lambda,k)=h_{n,\kappa}(\lambda,k)$ by~$\Lambda_{m,n,\kappa}$ (which is independent of $k$).
These numbers $\Lambda_{m,n,\kappa}$ are eigenvalues of $-\Delta_\Sp$ and they are given by
\begin{equation}\label{Lambda}
 \Lambda_{m,n,\kappa}=(2m+2n+|\kappa|)(2m+2n+|\kappa|+1),\qquad |\kappa|:=\kappa_1+\kappa_2+\kappa_3 .
 \end{equation}
An eigenfunction of $-\Delta_\Sp$ corresponding to the eigenvalue $\ell(\ell+1)$ is given by the sphero-conal harmonic
\begin{equation}\label{harmonic}
 U_{m,n,\kappa}(s,t)=V_{m,n,\kappa}(s)W_{m,n,\kappa}(t),
 \end{equation}
where $2m+2n+|\kappa|=\ell$, $W_{m,n,\kappa}(t)$ is a Lam\'e polynomial and
\[ V_{m,n,\kappa}(s)= W_{m,n,\kappa}(K+{\rm i}K'-{\rm i}s).\]
The Lam\'e polynomial $w(t)=W_{m,n,\kappa}(t)$ is a solution of the Lam\'e equation
\[ w''+\bigl(h_{n,\kappa}(\ell(\ell+1),k)-\ell(\ell+1)k^2 \sn^2(t,k)\bigr) w=0\]
of the form
\[ w(t)=\sn^{\kappa_1}(t,k)\cn^{\kappa_2}(t,k)\dn^{\kappa_3}(t,k) P\bigl(\sn^2(t,k)\bigr),\]
where $P$ is a polynomial of degree $m+n$. Moreover, $V_{m,n,\kappa}(s)$ has $m$ zeros in $(0,K')$ and $W_{m,n,\kappa}(t)$ has $n$ zeros in $(0,K)$.
For more details on Lam\'e polynomials, we refer to Arscott~\mbox{\cite[Chapter~IX]{A}}. Arscott distinguishes eight species of
Lam\'e polynomials which correspond to the eight parities $\kappa\in\{0,1\}^3$ in our notation.

\section{The Pr\"ufer transformation}\label{prufer}

We introduce the Pr\"ufer radius $r(t)>0$ and the Pr\"ufer angle $\theta(t)$ \cite[Section~27.IV]{W} for
system~\eqref{SL2} and~\eqref{bc2} by setting
\[ w(t)=r(t)\sin \theta(t),\qquad \frac{w'(t)}{q(t)}=r(t)\cos \theta(t). \]
Then $\theta(t)$ satisfies the first-order differential equation
\begin{equation}\label{prode}
\frac{{\rm d}\theta}{{\rm d}t}= q(t)\bigl(\cos^2\theta+\bigl(h-\lambda k^2\sn^2(t,k)\bigr)\sin^2\theta\bigr)
\end{equation}
with initial condition
\begin{equation}\label{prbc}
 \theta(0)=\tfrac12(1-\kappa_1)\pi .
 \end{equation}
The pair $(\lambda,h)$ lies on the $n$-th eigencurve if
\begin{equation}\label{prend}
 \theta(K)=\tfrac12(1+\kappa_2)\pi +n\pi .
\end{equation}

\begin{figure}[t]
\centering
\includegraphics[width=8cm]{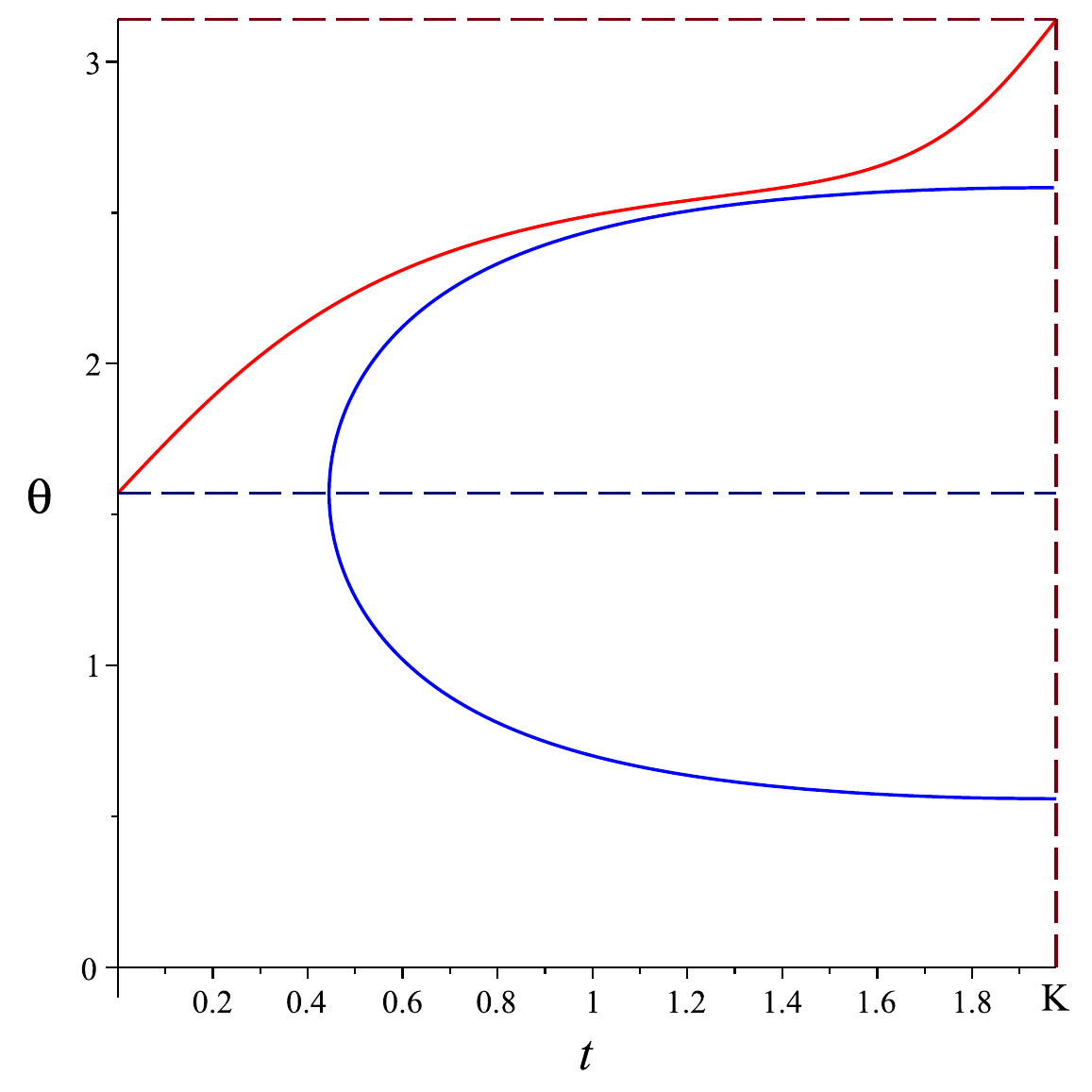}
\caption{The Pr\"ufer angle $\theta(t)$ (in red) and the curve on which the right-hand side of \eqref{prode}
vanishes (in blue).\label{fig2}}
\end{figure}

Consider the example $a=3$, $b=2$, $c=1$, $\kappa_1=0$, $\kappa_2=1$, $\lambda=5$ and $n=0$.
Then $h_{n,\kappa}(\lambda)=0.558216\dots$ and the graph of $\theta(t)$ is shown in Figure~\ref{fig2}.

\begin{Lemma}\label{prueferl1}
If $\kappa_2=1$, $\lambda>0$ and $n\in\N_0$, then every solution $\theta(t)$ of \eqref{prode}, \eqref{prbc}
and \eqref{prend} satisfies $\theta'(t)>0$ for $t\in[0,K]$.
\end{Lemma}
\begin{proof}
We know that $h>0$. If $h-\lambda k^2>0$ then $h-\lambda k^2\sn^2(t,k)>0$ for all $t\in[0,K]$ and the statement of the
lemma follows immediately from \eqref{prode}.
Now suppose that $h-\lambda k^2\le 0$. Then the region
\[ R:=\bigl\{(t,\theta)\in[0,K]\times \R\colon \cos^2\theta+\bigl(h-\lambda k^2\sn^2(t,k)\bigr)\sin^2\theta<0\bigr\}\]
is given by
\[ R=\{(t,\theta)\colon G(\theta)< t\le K\}, \]
where
\[ G(\theta)={\rm arcsn}\biggl(\frac{h+\cot^2\theta}{\lambda k^2}\biggr)^{1/2} .\]
This function is only defined for $\theta$ with $h+\cot^2\theta\le \lambda k^2$.
In Figure \ref{fig2}, $R$~is the region to the right of the blue curve.
Our goal is to show that the graph of $\theta(t)$ cannot enter the region $R$.

If the point $(t_1,\theta(t_1))$ lies on the boundary of $R$, then $\theta'(t_1)=0$ and so $(t,\theta(t))$ is outside
the closure $\bar R$ of $R$
for $t<t_1$ close to $t_1$. Therefore, once the graph of $\theta(t)$ enters $\bar R$ it must stay in~$\bar R$.
Since all points $(K,\ell \pi)$, $\ell\in\Z$, lie outside of $\bar R$, and $\kappa_2=1$,
the graph of $\theta(t)$ stays outside of $\bar R$ and so $\theta'(t)>0$ for all
$t\in[0,K]$.
\end{proof}

If $\kappa_2=0$, then Lemma \ref{prueferl1} is not always true. For example, if $\kappa_1=\kappa_2=0$
and $n=0$, then $\theta(0)=\theta(K)=\tfrac12\pi$, so $\theta'(t)$ cannot be positive for all $t\in[0,K]$.
If $\theta'(t)$ is not positive throughout the interval $[0,K]$ then arguing as in the proof of Lemma~\ref{prueferl1}
one can show that there is $t_0\in(0,K]$ such that $\theta'(t_0)=0$, $\theta'(t)>0$ for $t\in[0,t_0)$ and
$\theta'(t)<0$ for $t\in(t_0,K]$.

\begin{thm}%\label{pruefert1}
For all $m,n\in\N_0$ and all parities $\kappa=(\kappa_1,\kappa_2,\kappa_3)$ with $\kappa_2=1$, we have
\begin{equation}\label{ineq1}
 a^{-2}\Lambda_{m,n,\kappa}<\lambda_{m,n,\kappa}(a,b,c)< c^{-2}\Lambda_{m,n,\kappa},
 \end{equation}
where $\Lambda_{m,n,\kappa}$ is given by \eqref{Lambda}.
If $\kappa_2=0$, then
\begin{equation}\label{ineq2}
 \lambda_{m,n,\kappa}(a,b,c)<c^{-2}\Lambda_{m,n,\hat\kappa},
\end{equation}
where $\hat\kappa=(\kappa_1,1,\kappa_2)$.
If $\kappa_2=0$ and $m,n\ge 1$, then
\begin{equation}\label{ineq3}
 a^{-2}\Lambda_{m-1,n-1,\hat\kappa}<\lambda_{m,n,\kappa}(a,b,c) .
\end{equation}
\end{thm}
\begin{proof}
Since $\lambda_{m,n,\kappa}(ra,rb,rc)=r^{-2} \lambda_{m,n,\kappa}(a,b,c)$, it is sufficient to assume that $c=1$
in the proof of the second inequality in \eqref{ineq1}.
Then the function $q(t)$ defined by \eqref{q} satisfies $q(t)>1$.
Let $\psi(t)$ be the solution of the differential equation
\begin{equation*}%\label{prode2}
\frac{{\rm d}\psi}{{\rm d}t}= f(t,\psi):=\cos^2\psi+\bigl(h-\lambda k^2\sn^2(t,k)\bigr)\sin^2\psi
\end{equation*}
determined by the initial condition $\psi(0)=\frac12(1-\kappa_1)\pi$.
Then Lemma \ref{prueferl1} gives
\[\theta'(t)-f(t,\theta(t))=\theta'(t)\bigl(1-q(t)^{-1}\bigr)>0=\psi'(t)-f(t,\psi(t)).\]
We also have $\psi(0)=\theta(0)$ and $\psi'(0)<\theta'(0)$. Now the comparison theorem \RS{\cite[?????~II.9.III]{WW}} yields $\psi(t)<\theta(t)$
for $0<t\le K$.
Using this inequality for $t=K$, we obtain from \eqref{prend} that
\[ h_{n,\kappa}(\lambda;a,b,c)< h_{n,\kappa}(\lambda,k) .\]
Since $p(s)>1$ for $0<s\le K'$, we can show in a similar way that
\[ H_{m,\kappa}(\lambda;a,b,c)> H_{m,\kappa}(\lambda,k).\]
Let $\lambda_0=\lambda_{m,n,\kappa}(a,b,c)$. Then
\[ H_{m,\kappa}(\lambda_0,k)<H_{m,\kappa}(\lambda_0;a, b,c)=h_{n,\kappa}(\lambda_0;a,b,c)<h_{n,\kappa}(\lambda_0,k) .\]
Therefore, if we set $g(\lambda)=H_{m,\kappa}(\lambda,k)-h_{n,k}(\lambda,k)$, then $g(\lambda_0)<0$. Now $g(\Lambda_{m,n,\kappa})=0$ and $g(\lambda)$ is an increasing function so
$\lambda_0<\Lambda_{m,n,\kappa}$ completing the proof of the second inequality in \eqref{ineq1}.
The first inequality is proved in a very similar way.

To prove \eqref{ineq2}, we assume $\kappa_2=0$ and $c=1$. Then we have
\[ H_{m,\kappa}(\lambda;a,b,c)>H_{m,\hat\kappa}(\lambda;a,b,c),\qquad h_{n,\kappa}(\lambda;a,b,c)<h_{n,\hat\kappa}(\lambda;a,b,c),
\]
so $\lambda_{m,n,\kappa}(a,b,c)<\lambda_{m,n,\hat\kappa}(a,b,c)$.
Now \eqref{ineq2} follows from \eqref{ineq1} applied to $\hat\kappa$ in place of $\kappa$.

Inequality \eqref{ineq3} is proved in a similar way.
\end{proof}

\section{Approximation by matrix eigenvalue problems}\label{matrix}

The eigenvalue functions $h_n(\lambda)$ for equation \eqref{SL2} can be computed using the Pr\"ufer angle by the
SLEIGN2 code \cite{BEZ}. However, we present another method which approximates the eigenvalues $h_n(\lambda)$ by
eigenvalues of finite matrices.
We transform the differential equation \eqref{SL2} into trigonometric form and then expand
eigenfunctions in Fourier series.
This method was used by Ince \cite{I.FIP} in order to compute eigenvalues of the Lam\'e equation~\eqref{lame1}.

Following Erd\'elyi \cite[formula~(4), p.~56]{EMO3}, we substitute
\[ \tau=\tfrac12\pi-\am(t,k) \]
in \eqref{SL2} where $\am$ denotes the Jacobian amplitude function \cite[Section~22.16\,(i)]{dlmf}.
We multiply~\eqref{SL2} by $q(t)^3$ and use
\[ \frac{{\rm d}\tau}{{\rm d}t}=-\dn(t,k),\qquad \cn(t,k)=\sin \tau, \qquad\sn(t,k)=\cos\tau.\]
Then we obtain
\begin{equation}\label{ode7}
D w +\lambda C w= h B w,
\end{equation}
where $D$ is the differential operator
\begin{equation*}%\label{opA}
 Dw:=-\bigl(a^2\sin^2\tau+b^2\cos^2\tau\bigr)\bigl(1-k^2\cos^2\tau\bigr) \frac{{\rm d}^2w}{{\rm d}\tau^2}-c^2k^2\cos\tau\sin\tau \frac{{\rm d}w}{{\rm d}\tau}
\end{equation*}
and $C$, $B$ are multiplication operators
\begin{gather*}
 Cw := k^2\cos^2\tau \bigl(a^2\sin^2\tau+b^2\cos^2\tau\bigr)^2 w,\qquad
 Bw := \bigl(a^2\sin^2\tau+b^2\cos^2\tau\bigr)^2 w.
 \end{gather*}
The boundary conditions corresponding to \eqref{bc2} are
\begin{equation*}%\label{bc7}
 w'(0)=w'\bigl(\tfrac12\pi\bigr)=0,\qquad '=\frac{{\rm d}}{{\rm d}\tau} .
\end{equation*}
We compute the matrix representation of $D$ in the trigonometric basis $\cos(2n\tau)$, $n\in\N_0$.
Using
\begin{gather*}
\bigl(a^2\sin^2\tau+b^2\cos^2\tau\bigr)\bigl(1-k^2\cos^2\tau\bigr) =
\tfrac12\bigl(a^2+b^2\bigr)-\tfrac18k^2\bigl(a^2+3b^2\bigr) \\ \hphantom{\bigl(a^2\sin^2\tau+b^2\cos^2\tau\bigr)\bigl(1-k^2\cos^2\tau\bigr) =}{}\!
+\tfrac12\bigl(b^2-a^2-b^2k^2\bigr)\cos(2\tau)+\tfrac18k^2\bigl(a^2-b^2\bigr)\cos(4\tau),
\end{gather*}
we obtain
\begin{gather}
 D\cos(2n\tau)= n^2\bigl(2\bigl(a^2+b^2\bigr)-\tfrac12k^2\bigl(a^2+3b^2\bigr)\bigr)\cos(2n\tau) \nonumber\\ \hphantom{D\cos(2n\tau)=}{}
 +n^2\bigl(b^2-a^2-b^2k^2\bigr)(\cos((2n-2)\tau)+\cos((2n+2)\tau))\nonumber\\ \hphantom{D\cos(2n\tau)=}{}
 + \tfrac14n^2k^2\bigl(a^2-b^2\bigr)(\cos((2n-4)\tau)+\cos((2n+4)\tau))\nonumber \\ \hphantom{D\cos(2n\tau)=}{}
 +\tfrac12 nc^2k^2(\cos((2n-2)\tau)-\cos((2n+2)\tau)).\label{matrixD}
\end{gather}
Similarly,
\begin{gather*}
\bigl(a^2\sin^2\tau+b^2\cos^2\tau\bigr)^2\cos^2\tau =\tfrac1{16}\bigl(a^4+2a^2b^2+5b^4\bigr)+\tfrac1{32}\bigl(15b^4+2a^2b^2-a^4\bigr)\cos(2\tau)
\\ \hphantom{\bigl(a^2\sin^2\tau+b^2\cos^2\tau\bigr)^2\cos^2\tau =}{}
+\tfrac1{16}\bigl(b^2-a^2\bigr)\bigl(a^2+3b^2\bigr)\cos(4\tau) +\tfrac1{32}\bigl(a^2-b^2\bigr)^2\cos(6\tau)
\end{gather*}
gives
\begin{gather}
 C\cos(2n\tau)=
 \tfrac1{16} k^2\bigl(a^4+2a^2b^2+5b^4\bigr)\cos(2n\tau)\nonumber\\ \hphantom{ C\cos(2n\tau)=}{}
 +\tfrac1{64} k^2\bigl(15b^4+2a^2b^2-a^4\bigr)(\cos((2n-2)\tau)+\cos((2n+2)\tau))\nonumber\\ \hphantom{ C\cos(2n\tau)=}{}
 +\tfrac1{32} k^2 \bigl(b^2-a^2\bigr)\bigl(a^2+3b^2\bigr)(\cos((2n-4)\tau)+\cos((2n+4)\tau))\nonumber\\ \hphantom{ C\cos(2n\tau)=}{}
 +\tfrac1{64} k^2 \bigl(a^2-b^2\bigr)^2 (\cos((2n-6)\tau)+\cos((2n+6)\tau)).\label{matrixC}
\end{gather}
Finally,
\begin{gather*}
\bigl(a^2\sin^2\tau+b^2\cos^2\tau\bigr)^2 =\tfrac18\bigl(3a^4+2a^2b^2+3b^4\bigr)+\tfrac12\bigl(b^4-a^4\bigr)\cos(2\tau)+ \RS{3a^4\bigr)}\\
\hphantom{\bigl(a^2\sin^2\tau+b^2\cos^2\tau\bigr)^2 =}{}
+\tfrac18\bigl(a^2-b^2\bigr)^2\cos(4\tau)
\end{gather*}
gives
\begin{gather}
B\cos(2n\tau)=
 \tfrac18\bigl(3a^4+2a^2b^2+3b^4\bigr)\cos(2n\tau)\nonumber\\ \hphantom{B\cos(2n\tau)=}{}
 +\tfrac14\bigl(b^4-a^4\bigr) (\cos((2n-2)\tau)+\cos((2n+2)\tau))\nonumber\\ \hphantom{B\cos(2n\tau)=}{}
 +\tfrac1{16}\bigl(a^2-b^2\bigr)^2 (\cos((2n-4)\tau)+\cos((2n+4)\tau)).\label{matrixB}
\end{gather}

For $N\in\N$, we define $N+1$ by $N+1$ matrices $D_N$, $C_N$, $B_N$.
The entry in the $(j+1)$-th row and $(n+1)$-th column of $D_N$, $C_N$, $B_N$ is the coefficient of $\cos(2j\tau)$ on the right-hand side of~\eqref{matrixD}--\eqref{matrixB}, respectively, where $j,n=0,\dots,N$.
For instance, we have
\[ B_4=\begin{pmatrix} b_0 & b_1 & b_2 & 0 &0\\ 2b_1 & b_0+b_2 & b_1 & b_2 &0 \\
2b_2 & b_1 & b_0 & b_1&b_2 \\
0 & b_2 & b_1 & b_0 &b_1\\
0 & 0 & b_2 & b_1 & b_0
\end{pmatrix},\]
where
\[ b_0:=\tfrac18\bigl(3a^4+2a^2b^2+3b^4\bigr),\qquad b_1:=\tfrac14\bigl(b^4-a^4\bigr),\qquad b_2:=\tfrac1{16}\bigl(a^2-b^2\bigr)^2 .\]
We consider the generalized matrix eigenvalue problem
\begin{equation*}%\label{matrix1}
 D_N u+\lambda C_Nu= h B_N u.
 \end{equation*}
We denote its eigenvalues $h$ by $\hat h_{n,N}(\lambda;a,b,c)$, $n=0,1,\dots,N$, arranged in increasing order of their
real parts. We consider $\hat h_{n,N}(\lambda;a,b,c)$ as an approximation to $h_{n,\kappa}(\lambda;a,b,c)$
with parity $\kappa=(0,0,0)$. This is confirmed by the following result.

\begin{thm}\label{convergence}
For every $a>b>c>0$, $m,n\in\N_0$ and $\lambda>0$, we have
\[ \lim_{N\to\infty} \hat h_{n,N}(\lambda;a,b,c)=h_{n,\kappa}(\lambda;a,b,c),\qquad \text{where $\kappa=(0,0,0)$}.
\]
\end{thm}
\begin{proof}
We assume that $a=1$ without loss of generality. We apply a general theorem \cite[Theorem B.1]{V}
from operator theory in Hilbert space.
Our Hilbert space is $H=L^2\bigl(0,\tfrac12\pi\bigr)$. The operator on the left-hand side of \eqref{ode7}
is decomposed as $D+\lambda C=A+S$, $S=S_1+S_2+S_3$,
where
\begin{alignat*}{3}
& A u:=-u'',&&& \\
&S_1u:=f_1(\tau)u'',\qquad &&f_1(\tau):=1-\bigl(\sin^2\tau+b^2\cos^2\tau\bigr)\bigl(1-k^2\cos^2\tau\bigr),&\\
&S_2u:=f_2(t)u',\qquad&& f_2(\tau):=-c^2k^2\cos\tau\sin\tau,&\\
&S_3u:=f_3(\tau) u,\qquad&& f_3(\tau):=\lambda k^2\cos^2 \tau\bigl(\sin^2\tau+b^2\cos^2\tau\bigr)^2.&
\end{alignat*}
The domain of $A$ consists of all continuously differentiable functions $u\colon \big[0,\frac12\pi\big]\to\C$
such that~$u'$ is absolutely continuous, $u''\in H$, and $u'(0)=u'\bigl(\frac12\pi\bigr)=0$.
The operator $A$ is self-adjoint, positive semi-definite with compact resolvent. Its eigenfunctions are
\[ e_0(\tau)=\sqrt{\frac2\pi},\quad e_\ell(\tau)=\frac{2}{\sqrt\pi}\cos(2\ell\tau),\qquad \ell\in\N.\]
The system $\{e_\ell\}_{\ell\in\N_0}$ forms an orthonormal basis for $H$.
The domains of $S_1$, $S_2$ are taken to be~$D(A)$ and the domain of $S_3$ is $H$.

In \cite[Theorem B.1]{V}, it is assumed that there are nonnegative constants $\alpha$, $\beta$ with $\beta<1$ such~that
\begin{gather}\label{cond1}
\|Su\|^2\le \alpha^2\|u\|^2+\beta^2\|Au\|^2\qquad\text{for all} \quad u\in D(A),\\
|\langle P_nSu,u\rangle|\le \alpha \langle u,u\rangle+\beta\langle Au,u\rangle\qquad \text{for all} \quad u\in D(A),\label{cond2}
\end{gather}
where $P_n$ is the orthogonal projection onto the linear span of $\{e_0,e_1,\dots,e_n\}$,
and $\|\cdot\|$, $\langle \cdot,\cdot\rangle$ denote the norm and inner product of $H$.
The existence of these constants is shown in the following lemma.

The statement of Theorem \ref{convergence} now follows from \cite[Theorems B.1 and B.2]{V}.
\end{proof}

\begin{Lemma}
There are nonnegative constants $\alpha$, $\beta$ with $\beta<1$ such that
\eqref{cond1} and \eqref{cond2} hold.
\end{Lemma}
\begin{proof}
We show \eqref{cond1}.
We have
\begin{equation*}%\label{proof1}
\|S_1u\|\le \|f_1\|_\infty\|Au\|,\qquad \|S_3u||\le \|f_3\|_\infty \|u\|.
\end{equation*}
Moreover, for every $\epsilon>0$,
\begin{align*}
\|S_2u\|^2&=\int_0^{\pi/2} |f_2(\tau)|^2 |u'(\tau)|^2\,{\rm d}\tau\le \|f_2\|_\infty^2\int_0^{\pi/2} u'(\tau)\overline{u'(\tau)}\,{\rm d}\tau\\
&= -\|f_2\|_\infty^2\int_0^{\pi/2} u''(\tau)\overline{u(\tau)}\,{\rm d}\tau \le \|f_2\|_\infty^2\bigl(\epsilon^{-2}\|u\|^2+\epsilon^2\|Au\|^2\bigr).
\end{align*}
Therefore,
\[
\|Su\|\le \gamma\|u\|+\delta\|Au\|,\qquad \|Su\|^2\le \bigl(1+\epsilon^{-1}\bigr)\gamma^2\|u\|^2+(1+\epsilon)\delta^2\|Au\|^2,
\]
where
\[
\gamma= \|f_3\|_\infty+\epsilon^{-1}\|f_2\|_\infty,\qquad
\delta= \|f_1\|_\infty+\epsilon\|f_2\|_\infty.
\]
Since $\|f_1\|_\infty<1$, this implies \eqref{cond1} with $\beta<1$ if we choose $\epsilon>0$ sufficiently small.

We show \eqref{cond2}.
Let
\[ u=\sum_{\ell=0}^\infty c_\ell e_\ell \in D(A) .\]
Then
\[ P_n u(\tau)=\sum_{\ell=0}^n c_\ell e_\ell(\tau),\qquad \frac{{\rm d}}{{\rm d}\tau} P_n u(\tau)= -\sum_{\ell=1}^n 2\ell c_\ell\frac{2}{\sqrt\pi}\sin(2\ell \tau).\]
Therefore,
\begin{equation}\label{ineq}
 \|(P_n u)'\|^2=\sum_{\ell=1}^n 4\ell^2 |c_\ell|^2\le \sum_{\ell=1}^\infty 4\ell^2|c_\ell|^2=\|u'\|^2.
\end{equation}
Now
 \begin{align*}
\langle P_nS_1u,u\rangle=\langle S_1u,P_nu\rangle &{}=\int_0^{\pi/2} f_1(\tau)u''(\tau)\overline{P_n u(\tau)}\,{\rm d}\tau
 = -\int_0^{\pi/2} u'(\tau)\frac{{\rm d}}{{\rm d}\tau}\bigl( f_1(\tau)\overline{P_n u(\tau)}\bigr)\,{\rm d}\tau\\
&{}= - \int_0^{\pi/2} u'(\tau)f_1'(\tau)\overline{P_n u(\tau)}\,{\rm d}\tau-\int_0^{\pi/2} u'(\tau)f_1(\tau)\frac{{\rm d}}{{\rm d}\tau}\overline{P_n u(\tau)}\,{\rm d}\tau.
\end{align*}
Therefore, using \eqref{ineq}, for any $\epsilon>0$,
\begin{align*}
|\langle P_nS_1u,u\rangle|&\le \|f_1'\|_\infty \int_0^{\pi/2} |u'(\tau)||P_n u(\tau)|\,{\rm d}\tau +\|f_1\|_\infty \int_0^{\pi/2} |u'(\tau)|\biggl|\frac{{\rm d}}{{\rm d}\tau}P_n u(\tau)\biggr|\,{\rm d}\tau\\
&\le
\epsilon^{-1}\|f_1'\|_\infty\int_0^{\pi/2} |u(\tau)|^2\,{\rm d}\tau +
(\epsilon\|f_1'\|_\infty +\|f_1\|_\infty)\int_0^{\pi/2} |u'(\tau)|^2\,{\rm d}\tau\\
&= \epsilon^{-1}\|f_1'\|_\infty\langle u,u\rangle +(\epsilon\|f_1'\|_\infty+\|f_1\|_\infty) \langle Au,u\rangle .
\end{align*}
Moreover, we have
\begin{gather*}
|\langle P_n S_2u,u\rangle|\le \|f_2\|_\infty \bigl(\epsilon^{-1} \langle u,u\rangle +\epsilon \langle Au,u\rangle\bigr),\\
|\langle P_nS_3u,u\rangle|\le \|f_3\|_\infty \langle u,u\rangle.
\end{gather*}
By adding these estimates, we obtain \eqref{cond2}
with $\beta<1$ if we choose $\epsilon>0$ sufficiently small.
\end{proof}

Numerical example:
We take $a=3$, $b=2$, $c=1$. We compute the eigenvalue functions $H_0(\lambda)>H_1(\lambda)$ for system \eqref{SL1} and \eqref{bc1}
and $h_0(\lambda)<h_1(\lambda)$ for system \eqref{SL2} and \eqref{bc2} by the method from Theorem \ref{convergence} with $N=7$.
We choose $N=7$ because for larger values of $N$ (that should give us more accurate results) there is no change in the displayed values of the eigenvalues anymore. Currently we do not have strict error bounds for these computations.
The graphs of the eigenvalue functions are shown in Figure \ref{fig3}.
The $\lambda$-values of their intersection points give us eigenvalues of the Laplace--Beltrami operator $-\D$,
namely,
\[ \lambda_{0,0}=0,\qquad \lambda_{0,1}=1.074471\dots,\qquad \lambda_{1,0}=2.134154\dots,\qquad
\lambda_{1,1}=5.029767\dots .\]
The intersection points are computed using bisection or regula falsi (or one of its improved variants).

\begin{figure}[t]
\centering
\includegraphics[width=10cm]{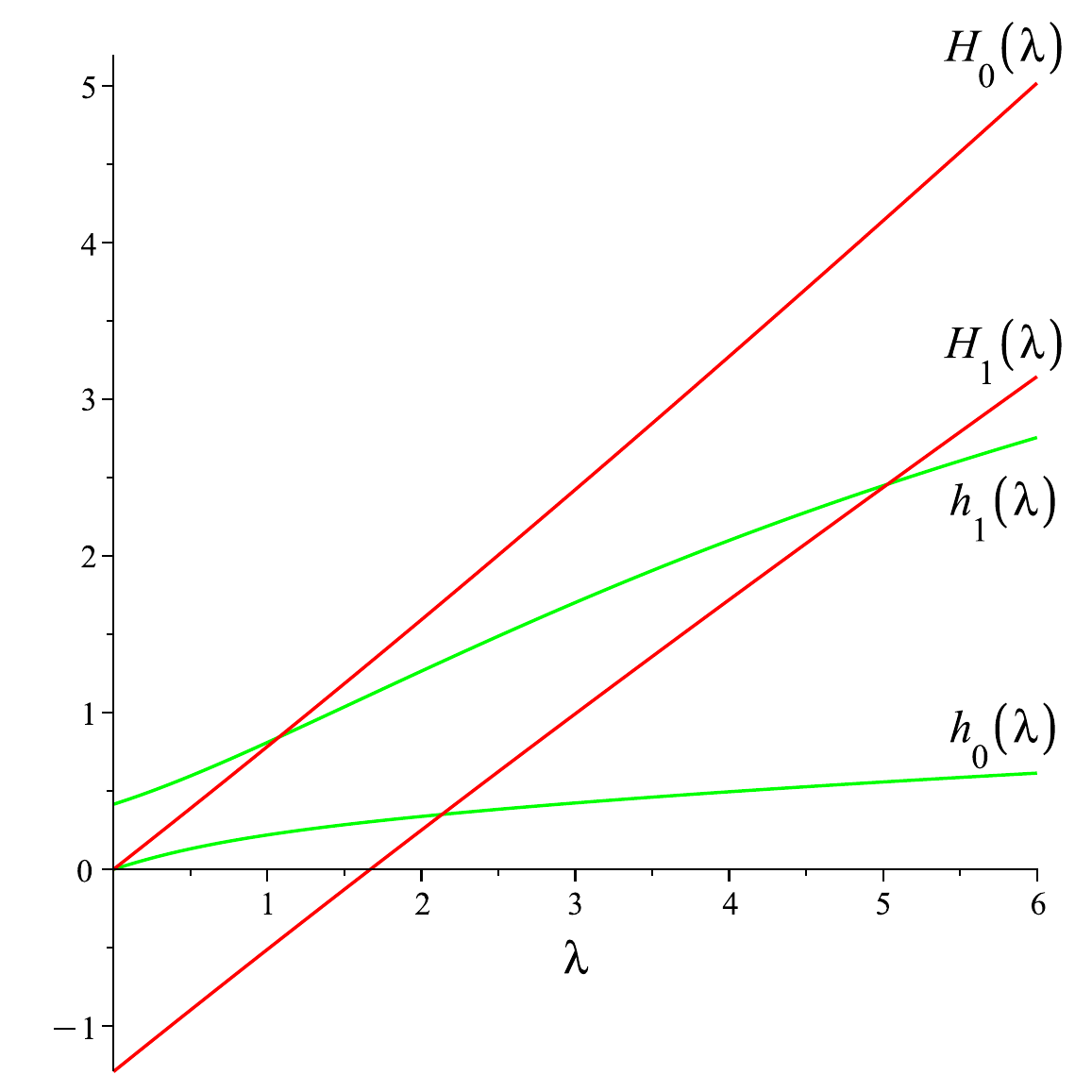}
\caption{Eigencurves $H_0(\lambda)$, $H_1(\lambda)$ (red) and $h_0(\lambda)$, $h_1(\lambda)$ (green) for $a=3$, $b=2$, $c=1$.
\label{fig3}}
\end{figure}

When working with the boundary conditions $w(0)=w\bigl(\frac12\pi\bigr)=0$ we use the basis $\sin(2n\tau)$, $n\in\N$.
We obtain the corresponding matrix representations for $D$, $C$, $B$ from \eqref{matrixD}--\eqref{matrixB}
by replacing $\cos$ by $\sin$ everywhere.
Note that these matrix representations agree with those with respect to $\cos(2n\tau)$ when we delete the
first row and first column of the latter, except for the three entries in the $i$-th row and $j$-th column
with $(i,j)=(1,1),(1,2),(2,1)$.
Similarly, when working with the boundary conditions
$w'(0)=w\bigl(\frac12\pi\bigr)=0$ we use the basis $\cos((2n+1)\tau)$,~${n\in\N_0}$.
The corresponding matrix representations are obtained from \eqref{matrixD}--\eqref{matrixB}
by replacing $n$ by $n+\frac12$. If the boundary conditions are $w(0)=w'\bigl(\frac12\pi\bigr)=0$, we use the basis $\sin((2n+1)\tau)$, $n\in\N_0$.
The corresponding matrix representations are obtained from \eqref{matrixD}--\eqref{matrixB}
by replacing $n$ by $n+\frac12$ and $\cos$ by $\sin$.
The matrix representations in the bases $\cos((2n+1)\tau)$ and $\sin((2n+1)\tau)$ only differ
in the six positions $(i,j)=(1,1), (2,1),(1,2),(3,1),(2,2),(3,1)$.

\section{Ellipsoids close to the unit sphere}\label{sphere2}

In this section, $k\in(0,1)$ is a fixed number and $k':=\sqrt{1-k^2}$. We consider the ellipsoid $\E(\epsilon)$ with semi-axes
\[ a=\bigl(1+k^2\epsilon\bigr)^{1/2},\qquad b=1,\qquad c=\bigl(1-k'^2\epsilon\bigr)^{1/2}\qquad \text{for} \quad 0<\epsilon<(k')^{-2}.\]
In the notation $\E(\epsilon)$, we suppressed the dependence of $\E(\epsilon)$ on $k$.
The given number $k$ agrees with the number $k$ from \eqref{k} associated with $\E(\epsilon)$.
If $\epsilon\to 0$, then $\E(\epsilon)$ approaches the unit sphere.
The functions $p(s)$ and $q(t)$ from \eqref{p} and \eqref{q} become
\[ p(s)=\bigl(1-k'^2\epsilon \cn^2(s,k')\bigr)^{1/2},\qquad q(t)=\bigl(1+k^2\epsilon \cn^2(t,k)\bigr)^{1/2} .\]
We denote the eigenvalue $\lambda_{m,n,\kappa}(a,b,c)$ for the ellipsoid $\E(\epsilon)$ by $\lambda_{m,n,\kappa}(\epsilon)$. Similarly, we denote the eigenvalue functions associated with \eqref{SL1}, \eqref{SL2}, \eqref{bc3}, \eqref{bc4},
\eqref{bc5} by
$H_{m,\kappa}(\lambda,\epsilon)$ and $h_{n,\kappa}(\lambda,\epsilon)$, respectively.
The corresponding two-parameter eigenvalue problems can be considered not only for
$0<\epsilon<k'^{-2}$ but also for $-k^{-2}<\epsilon<k'^{-2}$. If $\epsilon=0$, we obtain the eigenvalue problem determining Lam\'e polynomials
as mentioned in Section~\ref{sphere}. In particular, $\lambda_{m,n,\kappa}(0)=\Lambda_{m,n,\kappa}$
with $\Lambda_{m,n,\kappa}$ given by~\eqref{Lambda}.

\begin{Lemma}\label{derivatives1}
Let $m,n\in\N_0$ and $\kappa\in\{0,1\}^3$. The eigenvalue functions $H_{m,\kappa}(\lambda,\epsilon)$ and $h_{n,\kappa}(\lambda,\epsilon)$
are analytic at $\lambda=\Lambda_{m,n,\kappa}$ and $\epsilon=0$.
If $w(t)=W_{m,n,\kappa}(t)$ is a corresponding Lam\'e polynomial and $v(s)=w(K+{\rm i}K'-{\rm i}s)$, then
\begin{gather}
 \frac{\partial H_{m,\kappa}}{\partial \lambda}(\Lambda_{m,n,\kappa},0) =
\frac{\int_0^{K'} \dn^2(s,k')v(s)^2\,{\rm d}s}{\int_0^{K'}v(s)^2\,{\rm d}s}, \label{deriv1}\\
\frac{\partial H_{m,\kappa}}{\partial \epsilon}(\Lambda_{m,n,\kappa},0) =
\frac{k'^2\int_0^{K'} \cn(s,k') v(s)\frac{{\rm d}}{{\rm d}s}\bigl(\cn(s,k')\frac{{\rm d}v}{{\rm d}s}\bigr)\,{\rm d}s}{\int_0^{K'} v(s)^2\,{\rm d}s},\label{deriv2}\\
 \frac{\partial h_{n,\kappa}}{\partial \lambda}(\Lambda_{m,n,\kappa},0) =
\frac{k^2\int_0^K \sn^2(t,k)w(t)^2\,{\rm d}t}{\int_0^K w(t)^2\,{\rm d}t},\label{deriv3}\\
\frac{\partial h_{n,\kappa}}{\partial \epsilon}(\Lambda_{m,n,\kappa},0) =
\frac{k^2\int_0^K \cn(t,k) w(t)\frac{{\rm d}}{{\rm d}t}\bigl(\cn(t,k)\frac{{\rm d}w}{{\rm d}t}\bigr)\,{\rm d}t}{\int_0^K w(t)^2\,{\rm d}t}.\label{deriv4}
\end{gather}
\end{Lemma}
\begin{proof}
Consider a regular Sturm--Liouville problem of the form
\begin{equation}\label{SL}
 -\frac{{\rm d}}{{\rm d}t}\biggl(P(t,\mu)\frac{{\rm d}w}{{\rm d}t}\biggr)+ Q(t,\mu) w= h R(t,\mu) w,\qquad \alpha\le t\le \beta,
\end{equation}
with separated boundary conditions at $t=\alpha$ and $t=\beta$ (we use only Neumann or Dirichlet conditions.)
The eigenvalue parameter is $h$ and the perturbation parameter is $\mu$. The coefficient functions $P(t,\mu)$,
$Q(t,\mu)$ and $R(t,\mu)$ are continuous with continuous partial derivatives $P_\mu$, $Q_\mu$ and $R_\mu$ with respect to $\mu$
for $t\in[\alpha,\beta]$ and $\mu$ is some interval $J$. Let $h(\mu)$ be the eigenvalue of this Sturm--Liouville problem
determined by a fixed oscillation number of a corresponding eigenfunction $w(t,\mu)$. The eigenfunction $w(t,\mu)$
satisfies initial conditions at $t=\alpha$ that are independent of $\mu$.
We differentiate \eqref{SL} with respect to $\mu$, multiply by $w(t,\mu)$, and integrate from $t=\alpha$ to $t=\beta$. We obtain
\begin{gather}
 \frac{{\rm d}h}{{\rm d}\mu} \int_\alpha^\beta\! R(t,\mu) w(t,\mu)^2\,{\rm d}t
 =\int_\alpha^\beta\! \bigl(P_\mu(t,\mu)w'(t,\mu)^2+(Q_\mu(t,\mu)-hR_\mu(t,\mu))w(t,\mu)^2\bigr)\,{\rm d}t.\!\label{deriv}
\end{gather}
If we apply \eqref{deriv} to our eigenvalue problems \eqref{SL1} and~\eqref{SL2} with boundary conditions~\eqref{bc3}--\eqref{bc5} and $\mu=\lambda$, we
immediately obtain \eqref{deriv1} and \eqref{deriv3}.
If we apply \eqref{deriv} to \eqref{SL2}, \eqref{bc2}, \eqref{bc3} and $\mu=\epsilon$, then we obtain
\[ \frac{\partial h_{n,\kappa}}{\partial \epsilon}(\lambda,0)=
\frac{k^2\int_0^K \cn^2(t,k) \bigl(-w'(t)^2+\bigl(\lambda k^2\sn^2(t,k)-h\bigr)w(t)^2\bigr)\,{\rm d}t}{2\int_0^K w(t)^2\,{\rm d}t},\]
where $\lambda=\Lambda_{m,n,\kappa}$.
This gives \eqref{deriv4} using Lam\'e's differential equation
\[w''+\bigl(h-\lambda k^2\sn^2(t,k)\bigr)w=0\]
and the identity
\[ \cn^2(t,k)\bigl(-w'^2+w w''\bigr)=-\bigl(\cn^2(t,k)w w'\bigr)'+2\cn(t,k)w(\cn(t,k)w')'. \]
Equation \eqref{deriv2} is proved similarly.
\end{proof}

\begin{thm}%\label{derivatives2}
Let $m,n\in\N_0$ and $\kappa\in\{0,1\}^3$. The function $\lambda_{m,n,\kappa}(\epsilon)$ is analytic at $\epsilon=0$
with derivative
\begin{equation}\label{derivatives3}
 \lambda_{m,n,\kappa}'(0)=
-\frac{ \frac{\partial H_{m,\kappa}}{\partial \epsilon}(\Lambda_{m,n,\kappa},0)-\frac{\partial h_{n,\kappa}}{\partial \epsilon}(\Lambda_{m,n,\kappa},0)}{ \frac{\partial H_{m,\kappa}}{\partial \lambda}(\Lambda_{m,n,\kappa},0)-\frac{\partial h_{n,\kappa}}{\partial \lambda}(\Lambda_{m,n,\kappa},0)},
\end{equation}
where the four partial derivatives on the right are given in Lemma $\ref{derivatives1}$.
\end{thm}
\begin{proof}
The function $\lambda=\lambda_{m,n,\kappa}(\epsilon)$ solves
\[ H_{m,\kappa}(\lambda,\epsilon)-h_{n,\kappa}(\lambda,\epsilon) =0,\qquad \lambda(0)=\Lambda_{m,n,\kappa}.\]
It follows from \eqref{deriv1} and \eqref{deriv3} that the partial derivative of $H_{m,\kappa}(\lambda,\epsilon)-h_{n,\kappa}(\lambda,\epsilon)$ with respect to $\lambda$ at the point $(\Lambda_{m,n,\kappa},0)$ is positive.
Now the implicit function theorem for analytic functions shows that $\lambda(\epsilon)$ is analytic at $\epsilon=0$ and
also gives \eqref{derivatives3}.
\end{proof}

We introduce the differential operator
\[ Au=\frac{k'^2\cn(s,k')\frac{\partial}{\partial s}\bigl( \cn(s,k')\frac{\partial u}{\partial s}\bigr)-k^2\cn(t,k)\frac{\partial}{\partial t}\bigl( \cn(t,k)\frac{\partial u}{\partial t}\bigr)}
{\dn^2(s,k')-k^2\sn^2(t,k)}.
\]
Then \eqref{derivatives3} becomes
\[ \lambda'_{m,n,\kappa}(0)=-\frac{\langle Au,u\rangle_\Sp}{\langle u,u\rangle_\Sp},\]
where $u=U_{m,n,\kappa}$ is the sphero-conal harmonic \eqref{harmonic}, and
the inner product is
\[ \langle u_1,u_2\rangle_\Sp =\int_{-K'}^{K'}\int_{-2K}^{2K} \bigl(\dn^2(s,k')-k^2\sn^2(t,k)\bigr)u_1(s,t)u_2(s,t)\,{\rm d}t{\rm d}s.\]
Fix $\ell\in\N_0$ and $\kappa\in\{0,1\}^3$ such that $\ell-|\kappa|$ is nonnegative and even.
Let $V$ denote the vector space of spherical
harmonics homogeneous of degree $\ell$ with parity $\kappa$. The dimension $d$ of $V$ is $d=\frac12(\ell-|\kappa|)+1$.
Let $u_1,u_2,\dots,u_d$ be an orthonormal basis of $V$ with respect to the inner product $\langle\cdot,\cdot\rangle_\Sp$
consisting of sphero-conal harmonics.
It is easy to see that $\langle Au_i,u_j\rangle_\Sp=0$ if $i\ne j$, so the matrix $M_1$ with entries
$\langle Au_i,u_j\rangle_\Sp$ is diagonal and its diagonal entries are
$-\lambda_{m,n,\kappa}'(0)$ with $2m+2n+|\kappa|=\ell$.

The differential operator $A$ can be expressed in spherical coordinates and then agrees with
the operator $A_1$ in \cite{EK} with $\alpha=\frac12k^2$, $\beta=0$ and $\gamma=-\frac12k'^2$.
It is shown in \cite{EK} that the matrix~$M_2$ with entries $(\langle AU_i,U_j\rangle_\Sp)$ with $U_1,U_2,\dots,U_d$
denoting an orthonormal basis of $V$
derived from separation of variables in spherical coordinates is tridiagonal. This matrix $M_2$ is similar to
the diagonal matrix $M_1$. Therefore, the derivatives $-\lambda_{m,n,\kappa}'(0)$ with $2m+2n+|\kappa|=\ell$
are the eigenvalues of the matrix~$M_2$.
This shows that our formula~\eqref{derivatives3} is consistent with
\mbox{\cite[Theorem~4]{EK}}.

As an example, consider $\ell=2$ and $\kappa=(0,0,0)$.
Then the dimension of $V$ is two.
If $m=0$, $n=1$, the corresponding Lam\'e polynomial is \cite[p.~205]{A}
\[ w(t)=\sn^2(t,k)-\frac{1}{3k^2}\bigl(1+k^2+\sqrt{1-k^2k'^2}\bigr). \]
If $m=1$, $n=0$, we obtain the corresponding Lam\'e polynomial by replacing $+\sqrt{1-k^2k^2}$ by $-\sqrt{1-k^2k'^2}$.
We find from \eqref{derivatives3}
\begin{eqnarray*}
\lambda'_{0,1,(0,0,0)}(0)&=&2-4k^2-\tfrac87\sqrt{1-k^2k'^2},\\
\lambda'_{1,0,(0,0,0)}(0)&=&2-4k^2+\tfrac87\sqrt{1-k^2k'^2}.
\end{eqnarray*}
This result matches the result from \cite{EK}.

\subsection*{Acknowledgments}

The author thanks the referees whose remarks led to several improvements of the paper.

\pdfbookmark[1]{References}{ref}
\LastPageEnding

\end{document}